\documentclass[12pt,draftcls,onecolumn]{IEEEtran}

\usepackage{graphicx}
\usepackage{amsmath}
\usepackage{amsthm}
\usepackage{amssymb}
\usepackage{tikz}
\usepackage{pgfplots}
\usetikzlibrary{positioning,calc}
\usepackage{amsmath}
\usepackage{url}
\usepackage{hyperref}

\usepackage{mleftright}
\usepackage{todonotes}

\newcommand{\tr}{\operatorname{tr}}
\newcommand{\diag}{\operatorname{diag}}

\newcommand{\be}{\begin{equation}}
\newcommand{\ee}{\end{equation}}
\newcommand{\R}{\mathbb R}

\newcommand{\abs}{\operatorname{abs}}

\newtheorem{Theorem}{Theorem}

\newtheorem{Proposition}{Proposition}
\newtheorem{Lemma}{Lemma}
\newtheorem{Corollary}{Corollary} 

\newtheorem{Remark} {Remark}
\newtheorem{Example} {Example}

\title{Orthant-Monotonic Norms and Additive D-Stability\thanks{The work of RO was supported in part by the Air Force Office of Scientific Research, under award numbers FA9550-23-1-0175 and FA9550-25-1-0223, and in part by the Viterbi Fellowship, Technion-IIT. The research of MM is partially supported by a research grant from the Israeli Science  Foundation~(ISF).}     }
\date{May 2026}

\author{Ron Ofir  
  and Michael Margaliot
\thanks{RO  is    with the Department of Electrical Engineering,
Yale University, CT, USA.
MM is with the School of  ECE,
		and the Sagol School of Neuroscience, 
		Tel-Aviv University, Tel-Aviv~69978, Israel.
        Correspondence:  \texttt{ron.ofir@yale.edu}
        }%
}

\begin{document}

\maketitle

\begin{abstract}
Matrix measures induced by vector norms are widely used in contraction theory
of nonlinear dynamical systems. A natural and important robustness question is whether negativity of a matrix measure is preserved
under arbitrary   nonnegative diagonal damping. Matrix measures with this property have been called
admissible. We show that an induced matrix measure is admissible if and only if the underlying norm is
orthant-monotonic. Equivalently, these are precisely the induced matrix measures satisfying~$\mu(D) = \max_i
\{d_{ii} \}$ for every nonnegative diagonal matrix $D$. We further show that this class is maximal for uniform
preservation of contraction under nonnegative diagonal perturbations. The result gives a new geometric
characterization of admissibility and clarifies the role of orthant-monotonicity in additive $D$-stability and
diffusion-induced instability.
\end{abstract}


\begin{center} {\bf MSC classification numbers:} 
15A60; 47A55.
\end{center}

\section{Introduction}

Let~$\mathbb D^n_{\geq 0} $ be the set of~$n\times n$ diagonal matrices with nonnegative entries.
   Suppose that a matrix~$A\in\R^{n\times n}$
is Hurwitz. 
Then the  system~$\dot x=Ax$ is asymptotically  stable,
and we may naturally expect that 
\begin{align}\label{eq:shift}
\dot x=Ax-Dx, \text{ with } D \in\mathbb D^n_{\geq 0},
\end{align}
is also asymptotically  stable, as we merely  add an independent  dissipation term for each state variable.
In other words, we may expect that~$A-D$ is also Hurwitz for any~$D\in\mathbb D^n_{\geq 0}$.
However, this is not so. For example, the~$2\times 2$ matrix~$A=\begin{bmatrix}
    1&-3\\1&-2
\end{bmatrix}$
is Hurwitz (as~$\det(A)>0$ and~$\tr(A)<0$), yet for any~$d_2>1$  we have that~$\det (A-\diag(0,d_2)) = 1-d_2 <0$, so~$A-\diag(0,d_2)$ is not Hurwitz. In general, 
stability is a global property of the interaction pattern between the state-variables,  and 
not a monotone function of diagonal damping.

The term~$-Dx$ in~\eqref{eq:shift} models  independent dissipation (or leakage) in each state variable. This is 
common in many linear and nonlinear
models representing  for example 
natural death rate in population dynamics, molecular degradation in chemical reaction networks  and systems biology, drug clearance in pharmacokinetics,
damping in mechanical systems, and more. Often one expects that if~$D$ is  
large enough then the term~$-Dx$ dominates the other terms in the vector field and guarantees stability, but as we saw in the example above  this is not necessarily true.

A similar stability problem  arises in the context of systems with diffusion, as shown for the first time in a famous paper by A. M. Turing~\cite{Turing_diff_1952}. The theory is for systems of PDEs, but  following Sontag~\cite{sontag_systems_bio_book} we   present   the basic ideas using a set of ODEs. Consider two systems: 
\be\label{eq:lin_sync}
\dot x=Ax \text{ and } \dot z=Az , 
\ee
with~$A\in\R^{n\times n}$ Hurwitz, that  are coupled via  a diagonal diffusion  term, i.e.,  
\begin{align}\label{eq:adz}
\dot x=Ax  +D (z-x),\nonumber \\
\dot z=Az  +D (x-z),
\end{align}
with~$D \in\mathbb D^n_{\geq 0}$.
Define a new state-vector~$p:[0,\infty) \to\R^{2n}$ by  $p:=T\begin{bmatrix}
    x\\z
\end{bmatrix}
$, with~$T:=\begin{bmatrix}
    I_n&I_n\\-I_n&I_n
\end{bmatrix}
$. Then 
\[
\dot p = \begin{bmatrix}
    A&0\\0&A-2D
\end{bmatrix}p.
\]
Since~$A\in\R^{n\times n}$ is Hurwitz, the coupled system~\eqref{eq:adz} will be asymptotically stable iff~$A-2D$ is Hurwitz. Furthermore, since~$p=\begin{bmatrix} z+x& z-x \end{bmatrix}^\top$, the two systems in~\eqref{eq:lin_sync}
synchronize (i.e,~$ x(t)-z(t) \to 0$ as~$t\to\infty$) iff~$A-2D$ is Hurwitz. 

Diffusion terms appear in many dynamical  models representing for example migration in ecological networks,   kinetic interactions in chemical reaction networks, electrical or chemical interactions  between neurons, information flow in cooperative systems, 
and more.

The discussion above motivates the following definition:  
A matrix~$A \in \R^{n \times n}$ is called additively $D$-stable if~$A-D$ is Hurwitz for any~$D\in\mathbb D^n_{\geq 0 } $. According to Kushel~\cite{Unifying_Matrix_Stability_2019},  this notion  first appeared in~\cite{cross_3_types} (see also~\cite{LOGOFET200575}).

\begin{Example}\label{exa:22case}
If~$A\in\R^{2\times 2}$ then~$A-D$ is Hurwitz iff
\[
\tr(A-D)=a_{11}-d_1+a_{22}-d_2 <0
\]
and
\[
\det(A-D) = \det(A)+d_1d_2-d_1a_{22}-d_2a_{11} >0,
\]
so~$A\in\R^{2\times 2}$ will be additively~$D$-stable iff
\[
\tr(A)<0, \; 
\det(A)>0, \text{ and }  a_{11},a_{22}\leq 0. 
\]
\end{Example}

In general, the problem of characterizing  the class  of  additively  $D$-stable  matrices is open, except for low orders of~$n$~\cite{HERSHKOWITZ1992161}. In some special cases,
there are also several known necessary conditions and sufficient conditions for additive $D$-stability~\cite{Unifying_Matrix_Stability_2019,HERSHKOWITZ1992161}. For example, if~$A$ is Metzler (i.e.,~$a_{ij}\geq 0$ for all~$i\not= j$) then~$A$ is 
additively $D$-stable iff it is Hurwitz.   

Additive $D$-stability has found many applications in the stability analysis of continuous-time 
reaction-diffusion models 
(see, e.g.,~\cite{GE2009736} and the references therein). 
In this context, the
matrix~$A$ represents the
linearization of the reaction  dynamics at a steady-state, and additive $D$-stability rules out  
 the possibility of diffusion-driven instabilities.  

Li and Wang~\cite{Wang2001DiffusionDrivenInstability} 
introduced the related  notion of an \emph{admissible} matrix measure, namely, 
a matrix measure~$\mu:\R^{n\times n}\to\R$ such that
\be\label{eq:def_admi}
\mu(-D)\leq 0 \text{ for any } D\in\mathbb D^n_{\geq 0}, 
\ee
and used it to analyze  
additive D-stability. Indeed, 
\be\label{eq:uniform_cond}
\mu(A)<0, \text{ with } \mu  \text{ admissible},
\ee 
implies that 
for any~$D\in\mathbb D^n_{\geq0 }$ we have 
\begin{align}\label{eq:AminusD}
\mu(A-D)&\leq\mu(A)+\mu(-D) \nonumber\\
&\leq \mu(A)\nonumber\\
&<0,
\end{align}
where we used the fact that matrix measures are sub-additive.
This implies in particular
that~$A-D$ is Hurwitz, so~$A$ is additively $D$-stable.  However, the condition in~\eqref{eq:uniform_cond} is in fact stronger than additive $D$-stability, as it implies   negativity of~$\mu(A-D)$ for all~$D\in\mathbb D^n_{\geq 0}$ with respect to the \emph{same measure}~$\mu$.

Here, we give a new and  complete   characterization of such matrix measures by proving that a matrix measure is admissible if and only if it is    induced by an orthant-monotonic  norm.
This links the geometric 
notion of orthant-monotonicity with the stability-motivated notion of admissible matrix measures. 
Moreover, we show that the class of   matrix measures induced by orthant-monotonic norms
is the maximal set of matrix measures that are relevant for additive $D$-stability (see Theorem~\ref{thm:compl} below).

We use standard notation. 
Vectors [matrices] are denoted by small [capital] letters.
The non-negative [positive] orthant in~$\R^n$ is~$\R^n_{\geq 0}:=\{x\in\R^n : x_i\geq 0 ,i=1,\dots,n\}$
[$\R^n_{> 0}:=\{x\in\R^n : x_i >  0 ,i=1,\dots,n\}$].
For a matrix~$A\in\R^{n\times m}$, $\abs(A)$ is the 
matrix whose entries  are the absolute values of~$a_{ij}$ for all~$i,j$.
A matrix~$A\in\R^{n\times n}$ is called Hurwitz if every eigenvalue of~$A$ has a negative real part. It is called Metzler if all its off-diagonal entries are non-negative. 
For two matrices $A,B\in\R^{n\times m}$, we write~$A\leq B$ if~$a_{ij}\leq b_{ij}$ for all~$i,j$. 
We use~$I_n$ to denote the~$n\times n$ identity matrix. A matrix is called sign-diagonal if it is diagonal and  each diagonal entry is either~$1$ or~$-1$. 
The set of all~$n\times n$ positive [non-negative] diagonal matrices, that is, diagonal matrices with positive [non-negative] diagonal entries is denoted by~$\mathbb D^n_{>0}$ [~$\mathbb D^n_{\geq 0}$].

 Given a vector  norm~$|\cdot|:\R^n\to\R_{\geq0 }$, the induced matrix norm is
 \[
 \|A\|:=\max_{|x|=1}|Ax|,
 \] and the induced matrix measure (also called logarithmic norm~\cite{strom1975logarithmic}) is
 \[
\mu(A):=\lim_{\varepsilon\to 0^+}\frac{\|I_n+\varepsilon A  \| -1}{\varepsilon} .
\]
In other words,
 \be\label{eq:exap_mu}
\|I_n+\varepsilon A \| =1+\varepsilon\mu(A)+o(\varepsilon),
 \ee
and  also  
\[
\mu(A)=\lim_{\varepsilon\to 0^+} \frac{\ln(\|\exp(A\varepsilon)\|)}{\varepsilon}.
\]

Induced matrix norms and matrix measures play an important role in many fields of applied mathematics.
Examples include 
 numerical linear algebra~\cite{strom1975logarithmic},  
  the asymptotic analysis of nonlinear dynamical systems using contraction theory~\cite{sontag_cotraction_tutorial,LOHMILLER1998683,bullo_contraction,kordercont}, and control synthesis~\cite{Manchester2015ControlCM}.

A matrix measure
is sub-additive, i.e.~$\mu(A+B)\leq \mu(A)+\mu(B)$ for any~$A,B\in\R^{n\times n}$.
Also,~$s(A)\leq
\mu(A) \leq \|A\|$,
where~$s(A)$ denotes the largest real part of all
 the eigenvalues of~$A$, so~$\mu(A)<0$ implies that~$A$ is Hurwitz. 
 Also, 
\be\label{eq:mat_measure_I}
 \mu(A+cI_n)=\mu(A)+ c  \text{ for any } A\in\R^{n\times n}, c\in\R, 
 \ee
see e.g.~\cite{vid_desoer}.

 For~$p\geq 1$, 
the~$\ell_p$ norm of a vector $x\in\R^n$ 
is~$|x|_p : =\left (\sum_{i=1}^n (\abs(x_i))^p \right )^{1/p}$, and~$|x|_\infty:=\max_i \abs(x_i)$. For the~$\ell_1$, $\ell_2$, and~$\ell_\infty$ norms
there exist  simple closed-form expressions for the induced matrix norms and matrix measures, making them  particularly useful for  contraction
analysis, see Table~\ref{table:lp_norms}.  

\begin{table}[ht]
\centering
  \[
\begin{array}{|c|c|c|}
\hline
\text{Vector norm} & \text{Induced matrix norm} & \text{Induced matrix measure} \\ \hline

|x|_1=\sum_{i=1}^n \abs(x_i ) 
&
\|A\|_1=\max_{j} \sum_{i=1}^n \abs(a_{ij})
&
\mu_1(A)=\max_{j}\left(a_{jj}+\sum_{i\neq j}\abs(a_{ij})\right)
\\ \hline

|x|_2=\left(\sum_{i=1}^n \abs^2(x_i) \right)^{1/2}
&
\|A\|_2=\sqrt{\lambda_{\max}(A^{\top}A)}
&
\mu_2(A)=\lambda_{\max}\!\left(\frac{A+A^{\top}}{2}\right)
\\ \hline

|x|_{\infty}=\max_i \abs(x_i)
&
\|A\|_{\infty}=\max_{i} \sum_{j=1}^n \abs(a_{ij})
&
\mu_{\infty}(A)=\max_{i}\left(a_{ii}+\sum_{j\neq i}\abs(a_{ij})\right)
\\ \hline
\end{array} 
\]
\caption{$\ell_p$ norms, their induced matrix norms, and induced  matrix measures for~$p=1,2,\infty$.}\label{table:lp_norms}
\end{table}

The remainder of this paper  is organized as follows. The next section reviews some preliminary results that are used later on. Section~\ref{sec:main} 
describes our main results. The final section concludes and describes several directions for further research.

\section{Preliminaries}
In this section, we
review some known definitions and results that will be used later on. For more details and proofs, see~\cite[Chapter 5]{matrx_ana}.

\subsection{Monotonic and absolute norms} 
A   norm~$|\cdot| : \R^n \to \R_{\geq 0}$ is called~\emph{monotonic} if for any~$x,y\in\R^n$ we have 
\[
\abs(x_i)\leq \abs(y_i) \text{  for all } i \implies |x| \le |y|. 
\]

A   norm~$|\cdot| : \R^n \to \R_{\geq 0}$ is  called~\emph{absolute} if  
\[
|x| = | \abs(x) |  \text{ for any } x\in\R^n,
\]
  where~$\abs(x):=\begin{bmatrix}
      \abs(x_1)&\dots&\abs(x_n)
  \end{bmatrix}^\top$. 

  It is well-known that a vector norm is monotonic if and only if it is absolute~\cite{Bauer1961_mono_norms}. In other words, a monotonic norm   only depends  on the absolute values of the vector entries.  
  For example, it is clear from the definition that any~$\ell_p$ norm is absolute and thus monotone. Furthermore, if~$D\in\R^{n\times n}$ is diagonal and non-singular 
  and~$p\in[0,\infty]$
  then the scaled norm~$|x|_{p,D}:=|D x|$ is an absolute (and thus monotonic)  norm,  and the induced matrix measure is~$\mu_{p,D}(A)=\mu_p (DAD^{-1})$. 

Geometrically, a monotonic norm has a   unit ball that is convex, centrally symmetric (about the origin), and  invariant under sign changes.

\subsubsection{Order-preserving properties of monotonic norms}
If~$|\cdot|:\R^n\to\R_{\geq 0} $ is absolute (i.e., monotonic) then  the induced matrix norm and matrix measure enjoy some ``order-preserving'' properties that have found interesting applications (see for example~\cite[Proposition 2]{Ofir2022minimum}\cite{network_contractive}).

For example, for any~$A\in\R^{n\times n}$ and~$x\in\R^n$, we have 
\begin{align*}
|Ax| 
  & = | \abs(Ax)| \\
  &\leq  |\abs(A)\abs(x)|\\
  &\leq \|\abs(A)\||x|.
\end{align*}
and thus
\begin{align}\label{eq:A_vs_absA}
 \|A\|\leq \|\abs(A)\| \quad \text{for all }A\in\R^{n\times n}.
\end{align}
Also,  if~$0\leq A\leq B$ then 
\begin{align*}
    |Ax|&=| \abs(Ax)|\\
        &\leq |A \abs(x)|\\
        &\leq |B\abs(x)|\\
        &\leq \|B\| |\abs(x)|\\
        &=\|B\||x|,
\end{align*}
so we conclude that for a matrix norm induced by a monotonic norm, we have 
\be\label{eq:mat_norm_mom}
0\leq A\leq B \implies 
\|A\| \leq \|B\|.
\ee
  A similar 
  argument shows that if~$\mu$ is a   matrix measure induced by a monotonic norm,~$A,B\in\R^{n\times n}$ are Metzler  and~$A\leq B$ then
  $\mu(A)\leq \mu(B)$.

Monotonic norms also have a kind of ``Perron-Frobenius'' property. To explain  this, fix~$A\in\R^{n\times n}$ with~$A\geq 0$. Then
for any~$x\in\R^n$, we have 
\[
\abs(Ax)\leq A \abs(x), 
\]
  so
\begin{align}\label{eq:nohold}
    |Ax|=|\abs(Ax)| 
        \leq  |A\abs(x)|.
\end{align}
 This implies that 
\be\label{eq:maxrn}
A\geq 0 \implies \|A\|:=\max_{|x|=1}|Ax|=\max_{|x|=1,\; x\geq 0}|Ax|.
\ee
In other words, it is always possible to find a maximizing  vector~$x $   in the non-negative orthant.

  \subsection{Orthant-monotonic norms}
  A more general class  of norms   are orthant-monotonic norms~\cite{Gries1967} (also called   weak monotonic norms, see, e.g.~\cite{Lavric1999orthantmono}).
A   norm~$|\cdot| : \R^n \to \R_{\geq 0}$ is   
  \emph{orthant-monotonic} 
if for any~$x,y\in\R^n$ we have
  \[
  x_i y_i \ge 0 \text{ and }   |x_i| \le |y_i|  
 \text{ for all } i \implies  |x| \le |y|.
 \]
Note that inside one orthant, moving away from the origin coordinatewise means moving outward.
An orthant-monotone norm respects this order: going outward in every coordinate cannot reduce your distance from the origin as measured by the norm.

  \begin{Example}\label{exa:ortant_not_monotone}
      Consider the norm on~$\R^2$ defined by  
\begin{equation}\label{eq:exa_om}
    |x|: = \begin{cases}
        |x|_\infty , & \text{if }\: x_1x_2 \ge 0, \\
        |x|_1, & \text{if }\: x_1x_2 < 0.
    \end{cases}
\end{equation}
This norm is orthant-monotonic. 
The unit circle in  this norm is shown in Fig.~\ref{fig:norm}.
This norm is not absolute, as 
\begin{equation}\label{eq:counterexample}
 2 =   |\begin{bmatrix}
        1&  -1
    \end{bmatrix}^\top |\neq  |\begin{bmatrix}
        1&1
    \end{bmatrix}^\top|=1.
\end{equation}
Hence, 
this norm is  orthant-monotonic, but not monotonic.
  \end{Example}

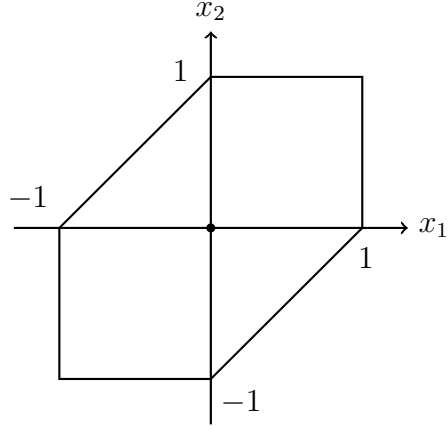
\begin{figure}
    \centering

\begin{tikzpicture}[scale=2]
\draw[->] (-1.3,0) -- (1.3,0) node[right] {$x_1$};
\draw[->] (0,-1.3) -- (0,1.3) node[above] {$x_2$};

\draw[thick,->] (-1.3,0) -- (1.3,0);
\draw[thick,->] (0,-1.3) -- (0,1.3);

\draw[thick]
(0,1) -- (1,1) -- (1,0) -- (0,-1) -- (-1,-1) -- (-1,0) -- cycle;
\foreach \p in { (0,0) }
\fill \p circle (0.03);

\node[right] at (0.9,-.2) {$1$};
\node[left]  at (-1,0.2) {$-1$};
\node[above] at (-.2,.9) {$1$};
\node[below] at (0.2,-1) {$-1$};

\end{tikzpicture}
 \caption{Unit ball    in  the orthant-monotonic norm defined in~\eqref{eq:exa_om}. }
    \label{fig:norm}
\end{figure}

Note that in the example above, $x:=\begin{bmatrix}
    1&-1
\end{bmatrix}^\top$ satisfies~$|\abs(x)| <|x|$. Taking~$A=I_2\geq0 $ shows that 
  for orthant-monotonic norms the key property in~\eqref{eq:nohold}
does not hold.

Geometrically, an orthant-monotonic  norm is characterized by a unit ball that is 
 convex, centrally symmetric (about the origin)  and whose boundary is monotone inside each orthant, but may have  different slopes in different orthants (and, therefore, the norm is not necessarily monotone).
  Orthant-monotonic norms have found  
 applications in
 approximation theory~\cite{orthant_mono_and_linear_programs},
 sparse optimization~\cite{Chancelier2021CapraConvex}, and norm-induced preferences in economic theory~\cite{Gershkov2020MonoNormsVoting}.

A norm~$|\cdot|:\R^n\to \R_{\geq 0}$ is orthant-monotonic if and only if 
\be\label{eq:ortmon_proj}
|\begin{bmatrix}
    x_1&\dots& &x_{j-1} &0&x_{j+1}&\dots& x_n
\end{bmatrix}^\top| 
\leq
|\begin{bmatrix}
    x_1&\dots &x_{j-1} &x_j&x_{j+1}&\dots& x_n
\end{bmatrix}^\top| 
\ee
for all~$x\in\R^n$ and all~$j\in\{1,\dots,n\}$. In other words,   projecting
 a vector on a coordinate axis can
only lead to  a smaller norm. For example, consider the norm on~$\R^2$ 
with unit ball defined by  
  the parallelogram with vertices~$\begin{bmatrix}
    2&2
\end{bmatrix}^\top$,
$\begin{bmatrix}
-2&-2
\end{bmatrix}^\top$,
$\begin{bmatrix}
    1&-1
\end{bmatrix}^\top$, and
$\begin{bmatrix}
    -1&1
\end{bmatrix}^\top$.
This norm is not orthant-monotonic, as
\[
\left | \begin{bmatrix}
    2\\ 0
\end{bmatrix}\right|>1=\left|\begin{bmatrix}
2\\2
\end{bmatrix}
\right  |.
\]

Summarizing the relations between the norms reviewed above, we have
\[
\ell_p \text{ norms }
\subset \text{ monotonic norms } = \text { absolute norms }\subset 
\text{ orthant-monotonic norms, }
\]
where~$\subset$ denotes strict inclusion.

\subsection{Induced matrix norms of  diagonal matrices}
Monotonic and orthant-monotonic  norms can also be characterized 
by the  effect of their induced matrix norms on diagonal matrices. 

An induced     
 matrix norm~$\|\cdot\|:\R^{n\times n}\to\R_{\geq 0}$ is
   monotonic   
 if and only if
 \[
 \|D\|=1 \text{ for any } D \text{ such that } \abs(D)=I_n,
 \]
(see~\cite{diag_transform_1988}), 
 and 
 if and only if 
 \[
  \|D\| = \max_i \{\abs(d_{ii})\}
 \text{ for any diagonal matrix } D  , 
 \]
   (see~\cite[Theorem 3]{Bauer1961_mono_norms}). 
   The second characterization  implies that if~$\mu$ is induced by a monotonic
   matrix measure and~$D\in\mathbb D^n_{\geq 0}$
then~$\mu(-D)=\max_i\{-d_{ii}\}\leq 0 $, so~$\mu$ is admissible~\cite{sontag2010contractive}.
In particular,  a matrix measure induced by a (diagonally scaled)~$\ell_p$ norm is admissible. 


An induced  matrix norm~$\|\cdot\|:\R^{n\times n}\to\R_{\geq 0}$ is   orthant-monotonic   norm if and only if
\begin{equation}\label{eq:bound_orthantmono}
    \|D\| = \max_i \{d_{ii}\} \text{ for any } D\in\mathbb D^n_{\geq 0}
\end{equation}
(see~\cite{Funderlic1979}).

\begin{Remark}
   In this paper, we only consider matrix norms and matrix measures that are induced by a norm in~$\R^n$. 
To shorten the notation,  from now on we say that a matrix norm or a  matrix measure 
is monotonic [orthant-monotonic] if it  is induced by a monotonic [orthant-monotonic] norm.
\end{Remark}

\section{Main results}\label{sec:main}

\subsection{Characterization of 
orthant-monotonic  matrix measures  }

Our first main result in this section provides a new  characterization of orthant-monotonic matrix measures and shows that this class coincides with the class of admissible matrix measures. introduced   in~\cite{Wang2001DiffusionDrivenInstability}.
\begin{Theorem}\label{thm:compl}
    Let~$\mu:\R^{n\times n}\to\R$ be a matrix measure. The following four conditions are equivalent. 
\begin{enumerate}
    \item \label{item:induced-orth_mono} $\mu$ is   orthant-monotonic;
  \item \label{item:admm} $\mu$ is admissible;
\item \label{item:max_dii_on_D_NON
NEGA} $\mu$ satisfies 
\[
\mu(D)=\max_i\{d_{ii}\} \text { for any } D\in \mathbb D^n_{\geq 0}; 
\]
  \item \label{item:nega
_on_A}
there exists a matrix~$A\in\R^{n\times n}$ such that~$\mu(A-D)<0$ for all~$D\in\mathbb D^n_{\geq 0}$.
\end{enumerate}
\end{Theorem}

\begin{IEEEproof}

  \noindent
  \ref{item:induced-orth_mono})$\implies$\ref{item:max_dii_on_D_NON
NEGA}):
  Suppose  that~$\mu$ is  orthant-monotonic. Fix~$D\in\mathbb D_{\geq 0}^n$. Then using~\eqref{eq:bound_orthantmono} gives 
    \begin{equation}
        \mu(D) = \lim_{h \to 0^+} \frac{\|I_n + hD\| - 1}{h} = \lim_{h \to 0^+} \frac{\max_i \{1 + hd_{ii}\} - 1}{h} = \max_i \{d_{ii}\}.
    \end{equation}

\noindent
\ref{item:max_dii_on_D_NON
NEGA})$\implies$\ref{item:induced-orth_mono}):
Suppose that~$\mu$   satisfies property \ref{item:max_dii_on_D_NON
NEGA})  in the  theorem. Fix~$D \in \mathbb D^{n}_{>0}$. Define  a diagonal matrix~$\ln( D )\in \R^{n \times n}$ by~$\ln( D):=\diag(\ln(d_{11}),\dots,\ln(d_{nn}))$. 
    Then
\[
\|D\|= \| \exp(\ln(D)) \| \leq \exp(\mu(\ln(D)))=\exp(\max_i \ln({d_{ii}}))=\max_i \{d_{ii}\},
\]
where the inequality follows from the fact that~$\|\exp(A)\|\leq \exp(\mu(A))$ for any~$A$ (see e.g.~\cite{strom1975logarithmic}).
  Since every~$d_{ii}$ is a positive eigenvalue of~$D$, we also have~$\|D\|\geq \max_i \{d_{ii}\}$. We conclude that~$\|D\|=\max_i \{d_{ii}\}$ for any~$D\in\mathbb D^n_{>0}$.  Now fix~$D\in\mathbb D^n_{\geq 0}$. Then  
   by continuity of norms,
    \begin{equation}
        \|D\| = \lim_{\varepsilon \to 0^+} \|D + \varepsilon I_n\| = \lim_{\varepsilon \to 0^+} \max_{i} \{d_{ii} + \varepsilon\} = \max_i \{d_{ii}\}.
    \end{equation}
    We conclude that~\eqref{eq:bound_orthantmono} holds, 
    and thus~$\mu$ is orthant-monotonic.

\noindent
\ref{item:max_dii_on_D_NON
NEGA})$\implies$\ref{item:admm}):
Fix~$D\in\mathbb D^n_{\geq 0}$. 
Using property \ref{item:max_dii_on_D_NON
NEGA}) gives
$
\mu(-D)=\max_i\{-d_{ii}\}\leq  0,
$
so~$\mu$ is admissible.

\noindent
\ref{item:admm})$\implies$\ref{item:max_dii_on_D_NON
NEGA}): 
Assume that~$\mu$ is admissible. Fix a diagonal matrix~$D\in\R^{n\times n}$. 
Let~$m:=\min_i \{d_{ii}\}$. 
Then~$D-mI_n \in \mathbb D^n_{\geq 0}$, and since~$\mu $ is admissible,
\[
\mu(-(D-m I_n))\leq 0 .
\]
Since~$-(D-m I_n)$ has an eigenvalue zero, we also have~$\mu(-(D-m I_n))\geq 0$, so~$\mu(-(D-m I_n))=0$. Thus,  
\begin{align*}
    0&=\mu(mI_n-D)\\
    &= m+\mu(-D),
\end{align*}
so
\be\label{eq:muofminusD}
\mu(-D)=-\min_i \{d_{ii}\}. 
\ee
Since~$D$ is an arbitrary diagonal matrix, 
we conclude, in particular, that~\ref{item:admm}) implies~\ref{item:max_dii_on_D_NON
NEGA}. 

\noindent \ref{item:admm})$\implies$\ref{item:nega
_on_A}):
Suppose that~$\mu$ is admissible. Then~$\mu(-I_n)=-1$, and~$\mu(-I_n-D) \le \mu(-I_n)+\mu(-D)<0$ for any~$D\in\mathbb D^n_{\geq 0}$, so property~\ref{item:nega
_on_A}) holds for~$A=-I_n$.

\noindent \ref{item:nega
_on_A})$\implies$\ref{item:admm}):
  Assume that property \ref{item:nega
_on_A}) holds. Note that applying this condition  with~$D=0$ gives~$\mu(A)<0$.  Fix~$D\in\mathbb D^n_{\geq 0}$. Then~$\mu(A-\alpha D)\leq\mu(A)+\mu(-\alpha D)<0$ for any~$\alpha>0$. Using 
    the  positive-homogeneity of matrix measures  yields
    \begin{align*}
       0&> \mu(A - \alpha D) \\&= \alpha\mu(\alpha^{-1}A-D).
    \end{align*}
    Taking~$\alpha>0$
    sufficiently large implies that~$0\geq \mu(-D)$, 
  so~$\mu$ is admissible.
  This completes the proof of Theorem~\ref{thm:compl}.
\end{IEEEproof}

  \begin{Remark}
      Note that property~\ref{item:nega
_on_A}) in Theorem~\ref{thm:compl} 
      implies that orthant-monotonic norms are the maximal class of norms for which contraction of 
$A-D$ is uniform over all 
$D\in\mathbb D^n_{\geq 0}$.
  \end{Remark}
%
%
 
\begin{Example}
Li and Wang~\cite{Wang2001DiffusionDrivenInstability}   showed that the matrix measure~$\mu_{\infty,T}:\R^{2\times 2}\to\R$ induced by the scaled $\ell_\infty$ norm~$|x|_{\infty,T}:=|Tx|_\infty$, with~$T:=\begin{bmatrix}
    1&2\\1&3
\end{bmatrix}$,
is not admissible. Indeed, for~$D=\diag(1,2)\in\mathbb D^2_{\geq 0}$, we have~$\mu_{\infty,T}(-D)=3$, so~\eqref{eq:def_admi}  does not hold. It follows from Theorem~\ref{thm:compl}  that the scaled norm~$|x|_{\infty,T}=\max\{\abs(x_1+2x_2),\abs(x_1+3x_2)\}$ is not orthant-monotonic. Indeed, 
\[
1= \left |\begin{bmatrix}
    2\\ -1
\end{bmatrix}  \right |_{\infty,T} <
\left |\begin{bmatrix}
     2\\ 0
\end{bmatrix} 
\right |_{\infty,T}=2  , 
\]
so~\eqref{eq:ortmon_proj} does not hold.
\end{Example}

The proof that~\ref{item:induced-orth_mono})$\implies$\ref{item:max_dii_on_D_NON
NEGA}) in  Theorem~\ref{thm:compl} above     holds for any diagonal matrix~$D$, and not just  for nonnegative diagonal matrices. This implies the following result.
\begin{Corollary}\label{coro:diago}
    If~$\mu$ satisfies any  of the equivalent conditions in Theorem~\ref{thm:compl}  then 
    \[
\mu(D)=\max_i\{d_{ii}\} \text{ for any diagonal matrix }   D\in\R^{n\times n}. 
\]
\end{Corollary} 

\subsection{Explicit bounds for    the matrix measure under diagonal perturbations}
We next show  that for orthant-monotonic matrix measures a 
diagonal perturbation  shifts the matrix measure in a controlled way.

\begin{Proposition}
    Let~$\mu:\R^{n\times n}\to\R$ be an orthant-monotonic matrix measure. Then for any~$A\in \R^{n\times n}$ and any diagonal matrix~$D\in\R^{n\times n}$, we have
    \be\label{eq:chattfp}
    \mu(A)-\max_i\{d_{ii}\}\leq \mu(A-D)\leq\mu(A)-\min_i\{d_{ii}\}.
    \ee
\end{Proposition}
Note that taking~$D=I_n$ and using~\eqref{eq:mat_measure_I} shows that the bounds are sharp.

\begin{IEEEproof}
First note that by Corollary~\ref{coro:diago},
\[
\mu(D)=\max_i\{d_{ii}\} \text{ and } 
\mu(-D)=-\min_i\{d_{ii}\}. 
\]
Thus,
\[
\mu(A-D)\leq\mu(A)+\mu(-D)=\mu(A)- \min_i\{d_{ii}\}.
\]
To prove the second bound in~\eqref{eq:chattfp}, note that
\begin{align*}
\mu(A)&=\mu(A-D+D)\\
&\leq \mu(A-D)+\mu(D)\\
&=\mu(A-D)+\max_i\{d_{ii}\},  
\end{align*}
and this completes the proof. 
\end{IEEEproof}
 
In the remainder of this section, we prove two more
properties of  orthant-monotonic measures.

Let~$\|\cdot\|_p$ denote  the matrix norm  induced by the~$\ell_p$ norm, which is absolute and thus monotonic. Using the explicit expression for~$\|\cdot\|_1$, $\|\cdot\|_2$, and~$\| \cdot \|_\infty$  (see Table~\ref{table:lp_norms}), it can be shown that all  these three norms satisfy~$\|A\|\geq \max\{\abs(a_{1,1}),\dots,
\abs(a_{n,n})\}$ for any~$A\in\R^{n\times n}$.  The next result generalizes this to  any orthant-monotonic  norm.

\begin{Lemma}\label{lem:matrix_nor}
Let~$|\cdot|:\R^{ n}\to\R_{\geq0 } $ be an
orthant-monotonic norm.   Then the induced matrix norm satisfies~$\|A\|\geq \max\{\abs(a_{1,1}),\dots,
\abs(a_{n,n})\}$ for any~$A\in\R^{n\times n}$.
\end{Lemma}
\begin{IEEEproof}
    Fix~$i\in\{1,\dots,n\}$.  Let~$e^i $ denote the $i$th unit vector in~$\R^n$.    
Then
    \begin{align*}
        \|A\||e^i| &\ge   |A e^i| \\
            &=   \left| \begin{bmatrix}
                a_{1,i} &
                \dots &
                a_{i-1,i} &
                a_{i,i} &
                a_{i+1,i} &
                \dots &
                a_{n,i}
            \end{bmatrix}^\top \right| \\
            &\ge \left| \begin{bmatrix}
                0 &
                \dots &
                0 &
                a_{i,i} &
                0 &
                \dots &
                0
            \end{bmatrix}^\top \right| \\
            &=  \abs(a_{i,i}) |e^i| ,
    \end{align*}
    where the inequality follows from~\eqref{eq:ortmon_proj}. We conclude that~$\|A\|\geq\abs(a_{i,i})$ for any~$i$. 
    \end{IEEEproof}

Let~$\mu_p$ denote  the matrix measure induced by the~$\ell_p$ norm, which is absolute and thus monotonic. Using the explicit expression for the~$\ell_1$, $\ell_2$, and~$\ell_\infty$ matrix measures (see Table~\ref{table:lp_norms}), it can be shown that if~$\mu_i(A)<0$ for~$i\in\{1,2,\infty\}$ then every diagonal entry of~$A$ is negative. The next result generalizes this to  any orthant-monotonic  norm. 

\begin{Proposition}\label{prop:orthant_stab_nec}
Let~$\mu:\R^{n\times n}\to\R $ be an
orthant-monotonic matrix measure, and let~$A\in\R^{n\times n}$. If~$\mu(A)<0$  then every diagonal entry of~$A$ is  negative.
\end{Proposition}

\begin{IEEEproof} 
   Since~$\mu(A)<0$, we have
    $
        \|I_n + hA\| < 1
    $
    for any  sufficiently small~$h>0$. 
Lemma~\ref{lem:matrix_nor} implies that~$\abs(1+h a_{i,i})<1$ for all~$i$,
    so~$a_{i,i}<0$ for all~$i$.
\end{IEEEproof}

Example~\ref{exa:22case} shows that a matrix~$A\in\R^{2\times 2}$ 
can be additively $D$-stable
even if one of its diagonal  entries is zero. Thus, the condition~$\mu(A)<0$, with~$\mu$ orthant-monotonic (i.e., admissible) is strictly stronger than the requirement that~$A$ is additively $D$-stable.
 
\section{Conclusion}
We proved that a matrix measure is  admissible  if and only if 
it is induced by an orthant-monotonic norm. This links the geometric notion of orthant-monotonicity with the stability-motivated notion of     admissible matrix measures.   We further showed that  
orthant-monotonic norms are the maximal class preserving contraction under arbitrary diagonal damping.

Questions for further research include the following.   
First, in proving contraction and, more generally,~$k$-contraction, it is often useful to use matrix measures induced by scaled norms (see, e.g.~\cite{sontag_cotraction_tutorial,RFM_entrain,kordercont}). If~$|\cdot|:\R^n\to\R_{\geq 0} $ is a norm and~$T\in\R^{n\times} $ is non-singular  then~$|x|_T:=|Tx|$ is a scaled norm. In our context, it is important to understand when a scaled norm is monotonic or orthant-monotonic. For some partial results along these lines, see~\cite{Lavric1997MonoNorms}.

As noted above, matrix norms induced by monotonic norms have some special properties over the set of nonnegative matrices
(see~\eqref{eq:mat_norm_mom} and~\eqref{eq:maxrn}). 
It is well-known that induced matrix norms have special minimality properties (see~\cite[Thm.~5.6.32]{matrx_ana}), and an interesting problem  
is to find necessary and sufficient conditions for a matrix norm to be minimal over the set of nonnegative matrices.

\bibliographystyle{IEEEtranS}

\begin{thebibliography}{10}
\providecommand{\url}[1]{#1}
\csname url@samestyle\endcsname
\providecommand{\newblock}{\relax}
\providecommand{\bibinfo}[2]{#2}
\providecommand{\BIBentrySTDinterwordspacing}{\spaceskip=0pt\relax}
\providecommand{\BIBentryALTinterwordstretchfactor}{4}
\providecommand{\BIBentryALTinterwordspacing}{\spaceskip=\fontdimen2\font plus
\BIBentryALTinterwordstretchfactor\fontdimen3\font minus \fontdimen4\font\relax}
\providecommand{\BIBforeignlanguage}[2]{{%
\expandafter\ifx\csname l@#1\endcsname\relax
\typeout{** WARNING: IEEEtranS.bst: No hyphenation pattern has been}%
\typeout{** loaded for the language `#1'. Using the pattern for}%
\typeout{** the default language instead.}%
\else
\language=\csname l@#1\endcsname
\fi
#2}}
\providecommand{\BIBdecl}{\relax}
\BIBdecl

\bibitem{sontag_cotraction_tutorial}
Z.~Aminzare and E.~D. Sontag, ``Contraction methods for nonlinear systems: A brief introduction and some open problems,'' in \emph{{Proc.\ 53rd IEEE Conf. on Decision and Control}}, Los Angeles, CA, 2014, pp. 3835--3847.

\bibitem{Bauer1961_mono_norms}
F.~L. Bauer, J.~Stoer, and C.~Witzgall, ``Absolute and monotonic norms,'' \emph{Numer. Math.}, vol.~3, pp. 257--264, 1961.

\bibitem{bullo_contraction}
\BIBentryALTinterwordspacing
F.~Bullo, \emph{Contraction Theory for Dynamical Systems}.\hskip 1em plus 0.5em minus 0.4em\relax Kindle Direct Publishing, 2022. [Online]. Available: \url{http://motion.me.ucsb.edu/book-ctds}
\BIBentrySTDinterwordspacing

\bibitem{Chancelier2021CapraConvex}
J.-P. Chancelier and M.~De~Lara, ``Capra-convexity, convex factorization and variational formulations for the $\ell_0$ pseudonorm,'' \emph{Set-Valued and Variational Analysis}, vol.~30, no.~2, p. 597–619, 2021.

\bibitem{cross_3_types}
G.~Cross, ``Three types of matrix stability,'' \emph{Linear Algebra Appl.}, vol.~20, pp. 253--263, 1978.

\bibitem{vid_desoer}
C.~A. Desoer and M.~Vidyasagar, \emph{Feedback Synthesis: Input-Output Properties}.\hskip 1em plus 0.5em minus 0.4em\relax Philadelphia: SIAM, 2009.

\bibitem{Funderlic1979}
R.~Funderlic, ``Some characterizations of orthant monotonic norms,'' \emph{Linear Algebra Appl.}, vol.~28, pp. 77--83, 1979.

\bibitem{GE2009736}
X.~Ge and M.~Arcak, ``A sufficient condition for additive \mbox{D-stability} and application to reaction–diffusion models,'' \emph{Systems Control Lett.}, vol.~58, no.~10, pp. 736--741, 2009.

\bibitem{Gershkov2020MonoNormsVoting}
A.~Gershkov, B.~Moldovanu, and X.~Shi, ``Monotonic norms and orthogonal issues in multidimensional voting,'' \emph{Journal of Economic Theory}, vol. 189, p. 105103, 2020.

\bibitem{Gries1967}
D.~Gries, ``Characterizations of certain classes of norms,'' \emph{Numer. Math.}, vol.~10, pp. 30--41, 1967.

\bibitem{HERSHKOWITZ1992161}
D.~Hershkowitz, ``Recent directions in matrix stability,'' \emph{Linear Algebra Appl.}, vol. 171, pp. 161--186, 1992.

\bibitem{matrx_ana}
R.~A. Horn and C.~R. Johnson, \emph{Matrix Analysis}, 2nd~ed.\hskip 1em plus 0.5em minus 0.4em\relax Cambridge University Press, 2013.

\bibitem{Unifying_Matrix_Stability_2019}
O.~Y. Kushel, ``Unifying matrix stability concepts with a view to applications,'' \emph{SIAM Review}, vol.~61, no.~4, pp. 643--729, 2019.

\bibitem{Lavric1997MonoNorms}
B.~Lavri{\v c}, ``Monotonicity properties of certain classes of norms,'' \emph{Linear Algebra Appl.}, vol. 259, pp. 237--250, 1997.

\bibitem{Lavric1999orthantmono}
------, ``A note on *orthant-monotonic norms,'' \emph{Linear Algebra Appl.}, vol. 299, no.~1, pp. 195--200, 1999.

\bibitem{LOGOFET200575}
D.~O. Logofet, ``Stronger-than-{Lyapunov} notions of matrix stability, or how ``flowers''' help solve problems in mathematical ecology,'' \emph{Linear Algebra Appl.}, vol. 398, pp. 75--100, 2005.

\bibitem{LOHMILLER1998683}
W.~Lohmiller and J.-J.~E. Slotine, ``On contraction analysis for non-linear systems,'' \emph{Automatica}, vol.~34, pp. 683--696, 1998.

\bibitem{diag_transform_1988}
G.~Loizou, ``Diagonal transformations,'' \emph{Applicationes Mathematicae}, vol.~20, no.~2, pp. 245--254, 1988.

\bibitem{Manchester2015ControlCM}
I.~R. Manchester and J.-J.~E. Slotine, ``Control contraction metrics: Convex and intrinsic criteria for nonlinear feedback design,'' \emph{IEEE Trans.\ Automat.\ Control}, vol.~62, no.~6, pp. 3046--3053, 2015.

\bibitem{RFM_entrain}
M.~Margaliot, E.~D. Sontag, and T.~Tuller, ``Entrainment to periodic initiation and transition rates in a computational model for gene translation,'' \emph{PLoS ONE}, vol.~9, no.~5, p. e96039, 2014.

\bibitem{Ofir2022minimum}
R.~Ofir, F.~Bullo, and M.~Margaliot, ``Minimum effort decentralized control design for contracting network systems,'' \emph{IEEE Control Systems Letters}, vol.~6, pp. 2731--2736, 2022.

\bibitem{network_contractive}
G.~Russo, M.~di~Bernardo, and E.~Sontag, ``A contraction approach to the hierarchical analysis and design of networked systems,'' \emph{IEEE Trans.\ Automat.\ Control}, vol.~58, no.~5, pp. 1328--1331, 2013.

\bibitem{sontag2010contractive}
E.~D. Sontag, ``Contractive systems with inputs,'' in \emph{Perspectives in Mathematical System Theory, Control, and Signal Processing: A Festschrift in Honor of Yutaka Yamamoto on the Occasion of his 60th Birthday}, J.~C. Willems, S.~Hara, Y.~Ohta, and H.~Fujioka, Eds.\hskip 1em plus 0.5em minus 0.4em\relax Berlin, Heidelberg: Springer, 2010, pp. 217--228.

\bibitem{sontag_systems_bio_book}
\BIBentryALTinterwordspacing
------, \emph{Lecture Notes on Mathematical Systems Biology}, 2026. [Online]. Available: \url{https://sontaglab.org/}
\BIBentrySTDinterwordspacing

\bibitem{strom1975logarithmic}
T.~Str{\"o}m, ``On logarithmic norms,'' \emph{SIAM J. Numerical Analysis}, vol.~12, no.~5, pp. 741--753, 1975.

\bibitem{Turing_diff_1952}
A.~M. Turing, ``The chemical basis of morphogenesis,'' \emph{Philosophical Transactions of the Royal Society of London Series B, Biological Sciences}, vol. 237, no. 641, p. 37–72, 1952.

\bibitem{Wang2001DiffusionDrivenInstability}
L.~Wang and M.~Y. Li, ``Diffusion-driven instability in reaction–diffusion systems,'' \emph{Journal of Mathematical Analysis and Applications}, vol. 254, no.~1, pp. 138--153, 2001.

\bibitem{kordercont}
C.~Wu, I.~Kanevskiy, and M.~Margaliot, ``$k$-contraction: theory and applications,'' \emph{Automatica}, vol. 136, p. 110048, 2022.

\bibitem{orthant_mono_and_linear_programs}
K.~Zietak, ``Orthant–monotonic norms and overdetermined linear systems,'' \emph{J. Approximation Theory}, vol.~88, no.~2, pp. 209--227, 1997.

\end{thebibliography}


\end{document}